\documentclass[a4paper,10pt]{article}
\usepackage{amsmath,amssymb}
\usepackage[english,frenchb]{babel}
\usepackage[colorlinks,citecolor=blue]{hyperref}
\usepackage[utf8]{inputenc}
\usepackage{graphicx}
\usepackage{euler} 
\usepackage{skak} 

\newcommand{\ZZ}{\mathbb{Z}}
\newcommand{\QQ}{\mathbb{Q}}

\newcommand{\prelie}{\operatorname{PreLie}}

\newcommand{\pl}{\curvearrowleft}

\newcommand{\pun}{\includegraphics[height=2.5mm]{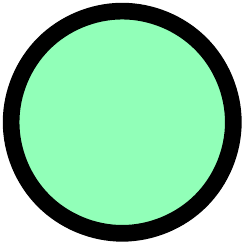}}

\newcommand{\corol}{\mathtt{Crl}}
\newcommand{\linear}{\mathtt{Lnr}}
\newcommand{\COR}{\textsc{Crls}}

\newcommand{\bop}{\textsc{B}_{+}}
\newcommand{\coul}{\operatorname{Color}}

\newcommand{\arb}[1]{\includegraphics[height=5mm]{a#1.pdf}}

\newcommand{\diam}{\diamond}

\newcommand{\sqx}{\sympawn}
\newcommand{\baro}{\overline{\Omega}}

\newcommand{\aut}{\operatorname{aut}}
\newcommand{\group}{\,\#\,}

\newtheorem{theorem}{Théorème}[section]
\newtheorem{proposition}[theorem]{Proposition}
\newtheorem{conjecture}[theorem]{Conjecture}

\newtheorem{lemma}[theorem]{Lemme}

\newtheorem{definition}{Définition}

\newtheorem{remark}[theorem]{Remarque}

\newenvironment{proof}{\begin{trivlist}\item{\bf{Preuve.}}}
  {\hfill\rule{2mm}{2mm}\end{trivlist}}

\title{Sur une série en arbres à deux paramètres}
\author{F. Chapoton\footnote{Ce travail a bénéficié du soutien du
    programme ANR CARMA.}}
\date{\today}

\begin{document}

\maketitle

\begin{abstract}
  On introduit une série en arbres $\sqx$, dont les coefficients sont
  des polynômes en $x$ à coefficients dans $\QQ(q)$. On considère les
  spécialisations de cette série en différentes valeurs de $x$, dont
  les $q$-entiers. Par une certaine limite en $x = -1/q$, on retrouve
  essentiellement la série en arbres $\Omega_q$ introduite dans un
  article antérieur, elle-même un $q$-analogue d'une série en arbres
  classique.
\end{abstract}

\selectlanguage{english}
\begin{abstract}
  One defines a new tree-indexed series $\sqx$, with coefficients that
  are polynomials in $x$ over the ring $\QQ(q)$. Several special
  evaluations of this series are obtained, in particular when $x$ is
  replaced by a $q$-integer. By taking a limit value when $x = -1/q$,
  one recovers the tree-indexed series $\Omega_q$ that was introduced
  in a previous article as a $q$-analog of a classical tree-indexed
  series $\Omega$.
\end{abstract}
\selectlanguage{french}

\section*{Introduction}

Les arbres enracinés sont des objets combinatoires simples et
classiques, qui sont apparus pour la première fois en mathématiques
dans un article de Cayley \cite{cayley} portant sur les champs de
vecteurs. Depuis lors, ils sont également intervenus
dans des domaines variés, notamment en analyse numérique dans les
travaux fondateurs de Butcher \cite{butcher} sur les méthodes de Runge-Kutta.

Des structures algébriques très riches en relation avec les arbres
enracinés ont été mises en évidence plus récemment, en particulier
plusieurs structures d'algèbres de Hopf, ainsi que les groupes qui
leur sont associés. On peut citer en particulier leur rôle dans le
point de vue de Connes et Kreimer sur la renormalisation
\cite{conneskreimer}.

La plupart de ces structures sont dérivées de la structure d'algèbre
pré-Lie libre sur les espaces vectoriels d'arbres enracinés
décorés. Au centre de tout ceci se trouve l'opérade $\prelie$, qui
décrit toutes les algèbres pré-Lie libres simultanément. La structure
d'opérade permet en particulier de définir une loi de groupe sur le
complété de l'algèbre pré-Lie libre sur un générateur, qui est un
analogue de la loi de composition des séries formelles. On appelle les
éléments de ce complété des séries en arbres.

Le sujet principal du présent article est une série en arbres
particulière $\sqx$ dont les coefficients dépendent de deux variables
$x$ et $q$. Cette série se place dans le contexte de plusieurs séries
en arbres plus simples, étudiées dans des articles antérieurs.

La première de ces séries en arbres est définie par
\begin{equation*}
  \sum_{T} \frac{1}{T!} \frac{T}{\aut(T)},
\end{equation*}
où la somme porte sur les arbres enracinés, $T!$ est une fonction
combinatoire simple appelée la factorielle de l'arbre $T$ et
$\aut(T)$ est la cardinal du groupe d'automorphismes de $T$. Cette série
apparaît naturellement comme solution formelle pour l'équation
différentielle ordinaire décrivant le flot d'un champ de vecteur. Elle
joue donc un rôle central dans la théorie de Butcher portant sur les
méthodes de Runge-Kutta.

Comme cette série est un élément du groupe des séries en arbres, on
peut considérer son inverse $\Omega$ dans ce groupe, qui joue
également un rôle en analyse numérique (``analyse rétrograde de
l'erreur''). La série en arbres $\Omega$ est un objet nettement plus
complexe, dont certains coefficients sont des nombres de Bernoulli.

Par le biais d'une relation entre $\Omega$ et une certaine famille
d'idempotents de Lie dans les algèbres de descentes des groupes
symétriques, une nouvelle série en arbres $\Omega_q$ a été définie
dans \cite{chap_omega}, dont les coefficients dépendent de la variable $q$. La
série $\Omega_q$ est un $q$-analogue de la série $\Omega$, au sens où
la spécialisation $\Omega_{q=1}$ redonne $\Omega$.

Une description du coefficient $\Omega_T$ d'un arbre enraciné $T$ dans
la série $\Omega$ a été obtenue dans \cite{wrightzhao}. Voici une
description équivalente, détaillée dans \cite{ogerM2}. On associe à
chaque arbre $T$ un polynôme $P_T$ en $x$, dont la valeur en $x=n$
compte les coloriages décroissants des sommets de $T$ par les entiers
de $0$ à $n$. Le coefficient $\Omega_T$ est alors donné par la dérivée
en $x=-1$ du polynôme $P_T$.

On obtient ici un $q$-analogue de ce résultat, qui donne une
description similaire du coefficient de $T$ dans la série $\Omega_q$,
ou plus exactement dans un légère variante $\baro_q$. Cette
description fait intervenir pour chaque arbre $T$ un polynôme en $x$ à
coefficients dans $\QQ(q)$, noté $\sqx_T$. La série en arbres $\sqx$
regroupe tous ces polynômes en un objet global, que l'on caractérise
par une équation fonctionnelle, et dont on étudie diverses
spécialisations.

Les coefficients de la série $\sqx$ admettent une interprétation
naturelle en termes de polytopes : ce sont des $q$-analogues de
polynômes d'Ehrhart, dont la théorie générale est développée dans
\cite{q_ehrhart}.

\medskip

L'article est organisé comme suit. La section \ref{notations}
introduit diverses notations sur les arbres et rappelle brièvement les
structures algébriques sur les séries en arbres qui seront
utilisées. La section \ref{principale} introduit l'équation
fonctionnelle qui caractérise la série en arbres $\sqx$, montre
qu'elle admet une unique solution et considère certains de ses
coefficients. La section \ref{evaluations} calcule les évaluations de
la série $\sqx$ pour différentes valeurs de $x$, et démontre en
particulier la relation voulue avec la série $\baro_q$. La section
\ref{autres} introduit l'opérateur de Hahn $\Delta$, décrit son
action sur les coefficients de $\sqx$ et démontre deux identités
ombrales remarquables sur les coefficients de $\baro_q$. La section
\ref{conjectures} présente deux conjectures portant sur les
coefficients de $\sqx$ et de $\baro_q$. La section \ref{auxiliaire}
démontre un résultat auxiliaire sur les séries en arbres. Enfin,
l'annexe \ref{debut} présente les premiers termes des séries en arbres
considérées.

Pour finir, quelques mots sur l'histoire de cet article, qui a
commencé par le calcul des coefficients des doubles
corolles\footnote{Arbres obtenus par greffe de copies
  de l'arbre à $2$ sommets sur une racine commune.}  dans la série
$\Omega$. Les coefficients des corolles dans la série $\Omega$ forment
la suite
\begin{equation*}
  1, -1/2, 1/6, 0, -1/30, 0, 1/42, 0, -1/30, 0, 5/66, 0, -691/2730, \dots
\end{equation*}
dans laquelle on reconnaît les célèbres nombres de Bernoulli. Pour les
doubles corolles, les premiers coefficients sont
\begin{equation*}
  1, 1/3, 1/30, -1/105, 1/210, -1/231, 191/30030, -29/2145, 2833/72930, \dots
\end{equation*}
Cette suite est bien moins connue, mais apparaît cependant dans la
littérature. Elle a été utilisée par Ramanujan, voir le numéro 9 du
chapitre 38 de \cite{berndtV}, et son prolongement par Villarino dans
\cite{villarino}. Elle est aussi étudiée par K. W. Chen dans
\cite{kwchen} sous un autre aspect.

Cette dernière référence fournit une interprétation ombrale à cette
suite de nombres. C'est en cherchant à généraliser cette description
ombrale aux mêmes coefficients dans la série $\Omega_q$ que cette
recherche a commencé.

\section{Notations}

\label{notations}

Dans tout l'article, $q$ sera une indéterminée.

On note $\Phi_d$ le polynôme cyclotomique d'ordre $d$, en la variable $q$.

On note $[n]_q$ le $q$-analogue de l'entier $n$ défini par
\begin{equation*}
  [n]_q = \frac{q^n-1}{q-1}.
\end{equation*}

Par \textbf{arbre}, on entend toujours un arbre enraciné, c'est-à-dire
un graphe fini connexe et sans cycle, muni d'un sommet distingué
appelé la racine.

On note $\# T$ le nombre de sommets d'un arbre $T$.

On appelle hauteur d'un sommet $s$ de l'arbre $T$ le nombre de sommets
dans l'unique chaîne reliant ce sommet $s$ à la racine de $T$.

La \textbf{hauteur d'un arbre} est le maximum des hauteurs de ses
sommets.

On oriente implicitement les arêtes en direction de la racine.

On appelle \textbf{feuille} les sommets qui n'ont pas d'arête entrante.

On note $\pun$ l'arbre à un sommet.

On appelle corolles les arbres de hauteur au plus $2$. Pour tout
$n\geq0$, on note $\corol_n$ la corolle à $n$ feuilles.

On appelle arbre linéaire et on note $\linear_n$ l'unique arbre à $n$
sommets de hauteur $n$.

Si $T_1,\dots,T_k$ sont des arbres, on note $\bop(T_1,\dots,T_k)$
l'arbre obtenu en greffant $T_1,\dots,T_k$ sur une nouvelle racine
commune.

\subsection{Structures algébriques}

\label{structalg}

On va utiliser dans cet article des structures algébriques assez
complexes sur les séries en arbres. On les rappelle très brièvement
ci-dessous, ainsi que certaines de leurs propriétés. Pour plus de
détails et pour les preuves, le lecteur pourra consulter
\cite{chap_omega} et \cite{idem_nara}.

On appelle \textbf{série en arbre} un élément du complété de l'algèbre
pré-Lie libre sur un générateur, par rapport à sa graduation
naturelle. Par la description connue de l'opérade $\prelie$ en termes
d'arbres enracinés \cite{chap_livernet}, une série en arbres est donc
une somme infinie d'arbres enracinés (sans étiquettes). Deux exemples
sont présentés explicitement dans l'appendice \ref{debut}. On note
$\pl$ le produit pré-Lie, qui est décrit sur les éléments de base par
la greffe : $S \pl T$ est la somme des greffes de l'arbre $T$ sur
l'arbre $S$ par ajout d'une arête entre la racine de $T$ et un sommet
de $S$.

Pour toute série en arbres $A$ et tout arbre $T$, on appelle
coefficient de $T$ dans $A$ et on note $A_T$ le coefficient dans le
développement
\begin{equation}
  \label{deva}
  A = \sum_{T} A_T \frac{T}{\aut(T)},
\end{equation}
où $\aut(T)$ est le cardinal du groupe d'automorphismes de $T$.

Pour toute série en arbres $A$, on note $A_n$ la restriction de la
somme \eqref{deva} aux arbres à $n$ sommets, c'est-à-dire la
composante homogène de degré $n$ de $A$.

On note $\Sigma_{\alpha}$ la suspension de paramètre $\alpha$, définie par
\begin{equation}
  \Sigma_{\alpha} \sum_{n\geq 1} A_n = \sum_{n \geq 1} \alpha^{n-1} A_n.
\end{equation}
Elle vérifie $\Sigma_{\alpha}\Sigma_{\beta}= \Sigma_{\alpha\beta}$.

On note $\COR$ la série en arbres particulière
\begin{equation}
  \COR = \sum_{n \geq 0} \frac{\corol_n}{n!}= \pun + \arb{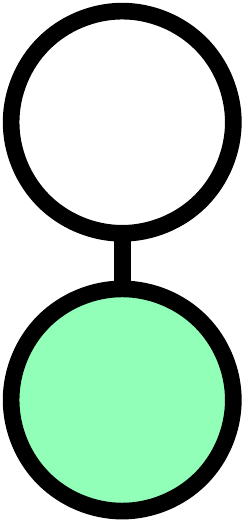} + \frac{\arb{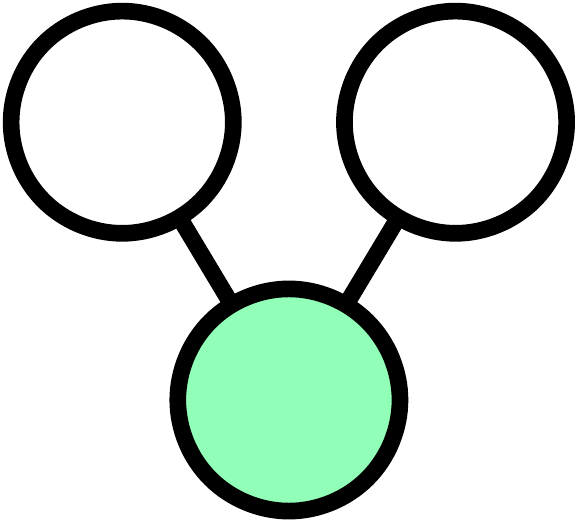}}{2} + \frac{\arb{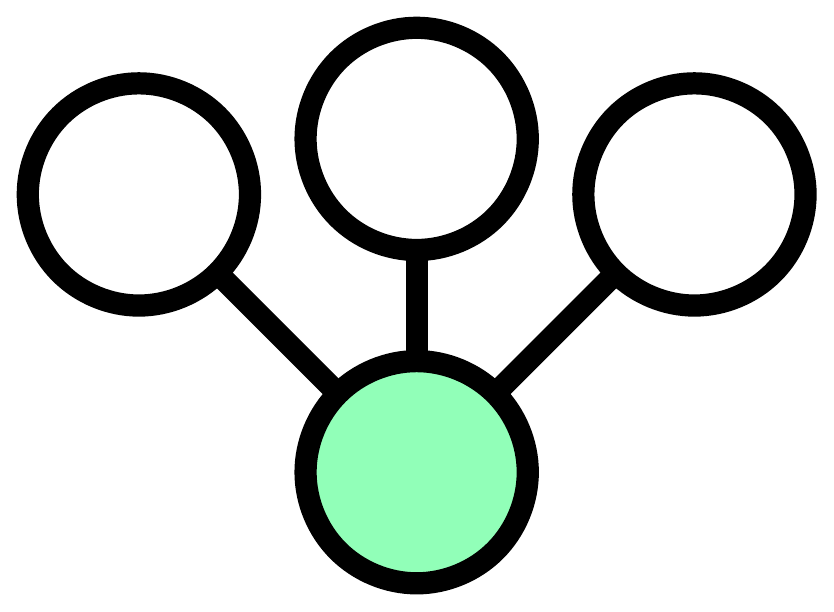}}{6} + \cdots.
\end{equation}

Sur l'espace des séries en arbres, il existe un produit associatif
$\circ$ (linéaire à gauche seulement) et un raffinement de cette
structure de monoïde sous la forme d'une opération $\diam$ à trois
arguments. Ces structures sont définies via l'opérade $\prelie$, voir
l'appendice de \cite{idem_nara}. Le produit $A \circ B$ formalise
l'insertion d'une série en arbres $B$ dans les sommets d'une série en
arbres $A$. L'opération $A \diam (B,C)$ correspond à l'insertion de
deux séries en arbres différentes, la série $B$ dans la racine et la
série $C$ dans les autres sommets de la série en arbres $A$. En
particulier, $A \diam (B,B) = A \circ B$.

L'opération $(A,B,C) \mapsto A \diam (B,C)$ est linéaire en $A$ et en
$B$. Elle vérifie
\begin{equation}
  \label{associatif}
  \COR \diam (\COR \diam (A,B),C) = \COR \diam (A, \COR \diam (B,C)).
\end{equation}

Lorsque les coefficients $A_T$ d'une série en arbres $A$ sont les
cardinaux d'ensemble finis associés aux arbres enracinés, la série $A$
est en quelque sorte une série génératrice, au sens usuel en
combinatoire, pour une certaine structure combinatoire sur les
arbres. Dans cette situation, on convient d'appeler $A$-structures sur
$T$ les éléments de l'ensemble associé à $T$.

Par exemple, il existe une unique $\COR$-structure sur chaque corolle,
et aucune sur les arbres qui ne sont pas des corolles.

Pour la série $U = \pun$, il existe une unique $U$-structure sur
l'arbre $\pun$, et aucune sur les autres arbres.

On verra dans la suite de l'article des exemples où les structures
sont des coloriages décroissants des sommets de $T$.

Soient $A$ et $B$ deux telles séries en arbres, séries
génératrices pour certaines $A$-structures et $B$-structures.

\begin{lemma}
  \label{lemme_ABC}
  La série en arbres $C$ définie par
  \begin{equation}
    C = \COR \diam (A,B)
  \end{equation}
  est la série génératrice des $C$-structures, où une $C$-structure
  sur $T$ est une triplet formé d'un sous-arbre $T_0$ contenant la
  racine, d'une $A$-structure sur l'arbre $T_0$, et d'une
  $B$-structure sur chacun des arbres qui sont les composantes
  connexes du complémentaire de $T_0$ dans $T$.
\end{lemma}

Par exemple, si $A$ et $B$ sont $\pun$, alors une $C$-structure unique
existe sur les corolles, et aucune n'existe sur les autres arbres. On
reconnaît les structures pour la série $\COR$. On a donc 
\begin{equation*}
  \COR \diam (\pun,\pun) = \COR,
\end{equation*}
ce qui résulte aussi du fait que $\pun$ est une unité pour $\circ$.

Par la suite, on utilisera souvent \textit{implicitement} cette
interprétation combinatoire de l'opération $\diam$ dans la série en
arbres $\COR$.

L'énoncé suivant est du même type, mais sans hypothèse sur la nature
combinatoire des séries concernées.

\begin{lemma}
  \label{lemme_corolle}
  Soient $A$ et $B$ deux séries en arbres telles que $A = \COR \diam (\pun, B)$. Alors
  \begin{equation}
    A_{\bop(T_1,\dots,T_k)}= \prod_{i=1}^{k} B_{T_i}.
  \end{equation}
\end{lemma}

La suspension agit sur le produit $\diam$ comme suit :
\begin{equation}
  \label{formule_susp}
  \Sigma_{\alpha} A \diam (B,C) = A \diam (\Sigma_{\alpha} B, \alpha \Sigma_{\alpha} C).
\end{equation}

\section{Définition principale}

\label{principale}

On introduit dans cette section une série en arbres $\sqx$, définie par
une équation fonctionnelle simple, et on en donne quelques propriétés
élémentaires.

L'équation fonctionnelle est la suivante :
\begin{equation}
  \label{main}
  q \Sigma_q \sqx - \COR \diam (\sqx, -\pun) = q(1+(q-1)x) \COR \diam (\pun, q\Sigma_q \sqx ) - \pun.
\end{equation}

On va montrer qu'elle admet une solution unique.

En considérant la composante homogène de degré $1$ de cette équation,
on obtient que le terme de degré $1$ de toute solution est $(1+q x) \pun$.

On peut écrire \eqref{main} sous la forme suivante :
\begin{equation}
  \label{recu_main}
  q \Sigma_q \sqx - \sqx = \underbrace{\left(\COR \diam (\sqx, -\pun) - \sqx\right)} + q(1+(q-1)x) \COR \diam (\pun, q\Sigma_q \sqx ) - \pun.
\end{equation}

Considérons la composante homogène de degré $n$ de cette égalité. Le
membre de gauche est exactement $q^n-1$ fois la composante de degré
$n$ de $\sqx$. Le membre de droite fait intervenir seulement des
composantes de $\sqx$ de degré strictement inférieur à $n$. En effet,
la différence (soulignée) entre les deux premiers termes du membre de
droite ne contient plus les composantes homogènes de $\sqx$ de degré
$n$. On peut donc calculer les composantes de $\sqx$ par
récurrence. Il en résulte que la solution de \eqref{main} existe et
est unique. Pour les premiers termes, voir l'appendice \ref{debut}.

On note que ce même raisonnement marche aussi pour toute
spécialisation de la variable $x$.

La relation \eqref{main} peut s'écrire élégamment sous la forme
\begin{equation}
  (-\pun) \group \sqx = (q\Sigma_q \sqx) \group \big{(}-q(1+(q-1)x)\pun\big{)},
\end{equation}
où le produit $\group$, défini par
\begin{equation}
  \label{BCH}
  x \group y = x + \COR \diam (y, x),
\end{equation}
est associatif. La formule \eqref{BCH} est un analogue dans le cadre
pré-Lie de la formule de Baker-Campbell-Hausdorff, voir par exemple
\cite{agracev} pour un énoncé équivalent (exprimé sans arbres).

\begin{proposition}
  Le coefficient $\sqx_T$ est le quotient d'un polynôme en $q$ et $x$
  par un produit de polynômes cyclotomiques en $q$.
\end{proposition}
\begin{proof}
  Ça résulte de la récurrence via la formule \eqref{recu_main}, car
  c'est vrai en degré $1$, et à chaque étape on divise par $q^n-1$ une
  combinaison linéaire de produits de coefficients précédents.
\end{proof}

\begin{proposition}
  Le coefficient $\sqx_T$ est un polynôme en $x$ de degré au plus $\#
  T$.
\end{proposition}
\begin{proof}
  Cette condition est vraie en degré $1$, et est préservée par la
  récurrence via la formule \eqref{recu_main}.
\end{proof}

\begin{figure}[h!]
  \centering
  \includegraphics[height=1.5cm]{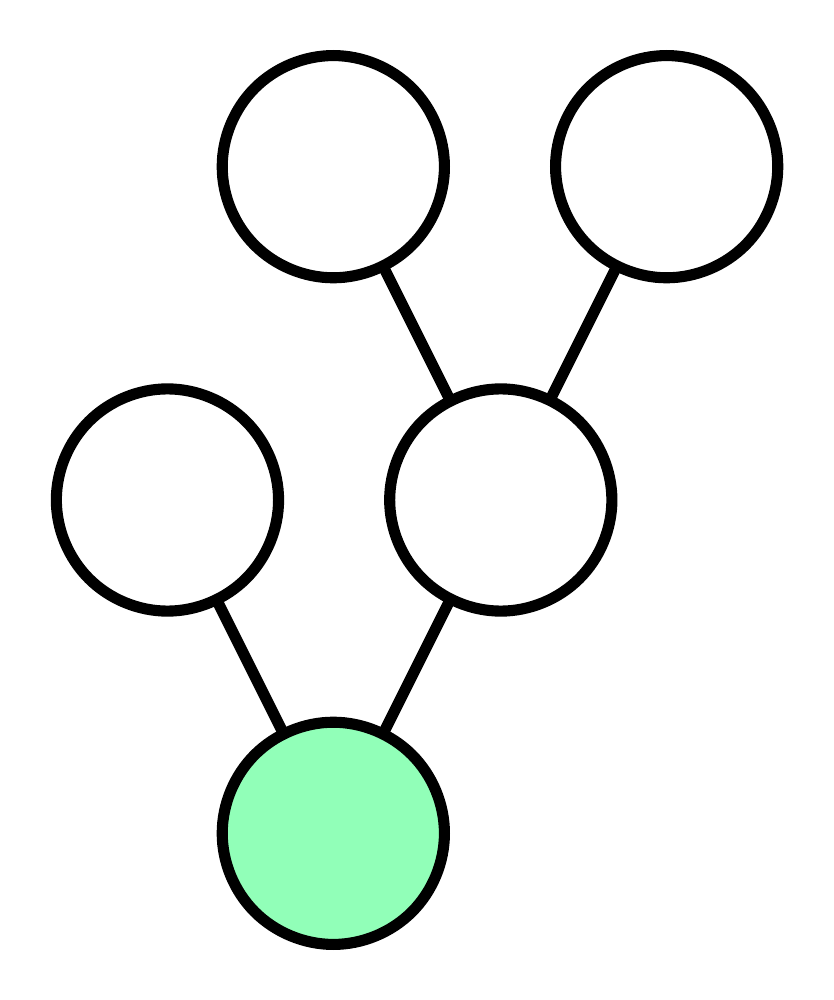}
 \caption{Un arbre $T$ sur $5$ sommets.}
  \label{fig:arbre5}
\end{figure}

Le numérateur du coefficient $\sqx_T$ pour l'arbre $T$ de la
figure \ref{fig:arbre5} est
\begin{multline*}
  (q x + 1) (q^2 x + q + 1) (q^3 x + q^2 + q + 1) \\((q^9 + 2 q^8 + 2 q^7+ 2 q^6 + q^5 )x^2 + (2 q^8  + 4 q^7  + 5 q^6  + 6 q^5+ 6 q^4  + 3 q^3 + q^2 ) x\\ + q^7 + 2 q^6 + 3 q^5 + 4 q^4  + 4 q^3 + 3 q^2 + 2 q + 1)
\end{multline*}
et son dénominateur est $[2]_q [3]_q [4]_q [5]_q$.

\begin{remark}
  Les coefficients du numérateur de $\sqx_T$ ne sont pas toujours des
  entiers positifs. Le plus petit exemple ayant des coefficients
  négatifs est fourni par l'arbre à $10$ sommets de la figure
  \ref{fig:contrex10}. C'est aussi le cas pour les corolles ayant au
  moins $12$ sommets.
\end{remark}

\begin{figure}[h!]
  \centering
  \includegraphics[height=1cm]{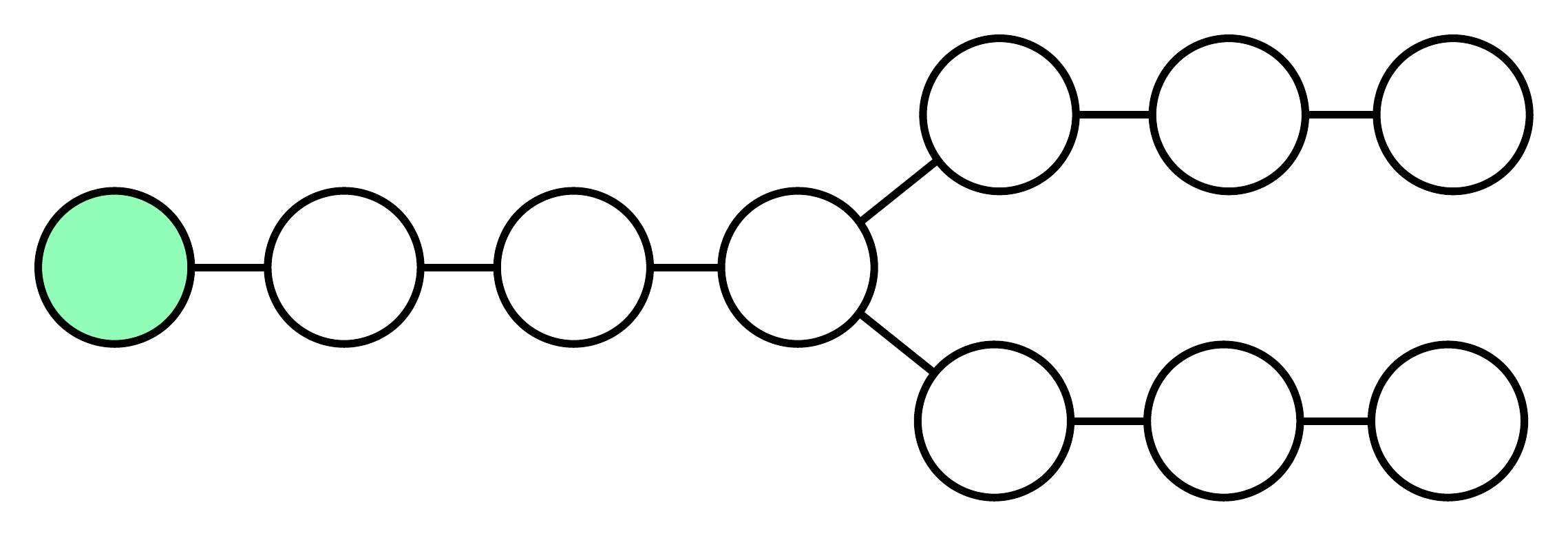}
  \caption{Contre-exemple avec $10$ sommets.}
  \label{fig:contrex10}
\end{figure}

\subsection{Cas des corolles}

Par les méthodes décrites dans \cite{chap_omega}, en utilisant un
quotient de l'opérade $\prelie$, on déduit de l'équation \eqref{main}
une équation
\begin{equation}
  q G(q t) - e^{-t} G(t) = q(1+(q-1)x) e^{q(1+q x)t} - 1
\end{equation}
pour la série génératrice exponentielle
\begin{equation}
  G(t) = \sum_{n \geq 0} \sqx_{\corol_n} \frac{t^n}{n!}.
\end{equation}

Par exemple, on obtient pour la corolle à trois feuilles le coefficient
\begin{equation*}
  \frac{ (q x + 1)  (q^2x + q + 1)  (q^4 \Phi_3 x^2 + (2 q^5 + 2 q^4  + 3 q^3 + 2 q^2) x + \Phi_3 \Phi_4) }{\Phi_2 \Phi_3 \Phi_4}
\end{equation*}

Il semble que les coefficients des corolles aient des dénominateurs
simples et réguliers.

\begin{conjecture}
  \label{conj_denominateur}
  Le dénominateur de $\sqx_{\corol_n}$ est le produit $\prod_{d=2}^{n+1} \Phi_d$.
\end{conjecture}

On a vérifié cette conjecture pour $n \leq 25$.

\subsection{Cas des arbres linéaires}

Par les méthodes décrites dans \cite{chap_omega}, en utilisant le
morphisme standard de l'opérade $\prelie$ vers l'opérade associative,
on déduit de l'équation \eqref{main} l'équation suivante :
\begin{equation}
  \label{recu_linear}
  G(q t) - (1-t) G = q(1+(q-1)x) t (1+G(q t)) -t,
\end{equation}
pour la série génératrice ordinaire
\begin{equation}
  G(t) = \sum_{n \geq 1} \sqx_{\linear_n} t^n.
\end{equation}

\begin{proposition}
  Le coefficient de l'arbre linéaire à $n$ sommets est donné par
  \begin{equation}
    \label{formule_linear}
    \sqx_{\linear_n} = (1+q x) \prod_{i=2}^{n} \frac{[i]_q+q^i x}{[i]_q}.
  \end{equation}
\end{proposition}
\begin{proof}
  En effet, en degré $1$ on trouve bien $1+q x$. La vérification que la
  formule \eqref{formule_linear} satisfait la récurrence correspondant
  à l'équation \eqref{recu_linear} est un calcul sans difficulté.
\end{proof}

\section{Évaluations diverses}

\label{evaluations}

On considère diverses évaluations de la série $\sqx$. En particulier,
on relie les valeurs lorsque $x$ est un $q$-entier $[n]_q$ à
l'énumération des coloriages décroissants, au sens large pour $n$
positif et au sens strict pour $n$ négatif.

\subsection{Valeur en $x=0$}

L'équation \eqref{main} devient en $x=0$ l'équation
\begin{equation}
  \label{main_0}
  q \Sigma_q Z - \COR \diam (Z, -\pun) = q \COR \diam (\pun, q\Sigma_q Z) -\pun,
\end{equation}
pour la série $Z = \sqx\mid_{x=0}$.

Soit $E$ la série définie par
\begin{equation}
  \label{defi_E}
  E = \sum_{T} \frac{T}{\aut(T)}.
\end{equation}

La série $E$ est une série génératrice (au sens de la section \ref{structalg}),
avec une unique $E$-structure sur chaque arbre enraciné.

Alors on a, par le lemme \ref{lemme_ABC}, 
\begin{equation}
  \label{prop_E}
  E = \COR \diam (\pun, E).
\end{equation}
On en déduit par la formule de suspension \eqref{formule_susp} que
\begin{equation}
  \label{auxi_E_1}
  q \Sigma_q E = q \COR \diam (\pun, q\Sigma_q E).
\end{equation}

Par ailleurs, la proposition \ref{prop_gen} et \eqref{prop_E} entraînent que
\begin{equation}
  \label{auxi_E_2}
  \COR \diam (E, -\pun) = \pun.
\end{equation}

On déduit des égalités \eqref{auxi_E_1} et \eqref{auxi_E_2} que $E$ est solution de
\eqref{main_0}.

Par unicité de la solution, on obtient
\begin{proposition}
  La valeur de $\sqx$ en $x=0$ est la série $E$ définie par \eqref{defi_E}.
\end{proposition}

\subsection{Valeur en $x=[n]_q$ pour $n$ positif ou nul}

Pour tout entier $n \in \ZZ$, l'équation \eqref{main} devient en
$x=[n]_q$ l'équation
\begin{equation}
  \label{main_1}
  q \Sigma_q Z - \COR \diam (Z, -\pun) = q^{n+1} \COR \diam (\pun, q\Sigma_q Z ) - \pun,
\end{equation}
pour la série $Z = \sqx\mid_{x=[n]_q}$.

On remarque que le cas $n= -1$ est particulier, l'unique solution
étant identiquement nulle. Le cas $n=0$ est celui du paragraphe précédent.

On suppose dans cette section que $n$ est positif ou nul.

Pour chaque arbre $T$, on définit un polynôme en $q$ comme suit. Soit
$\leq$ la relation d'ordre partiel naturelle sur les sommets de $T$, avec la
racine pour minimum.

\begin{definition}
  Un $n$-coloriage décroissant de $T$ est une application
  décroissante $c$ de $T$ dans l'ensemble $\{0,\dots,n\}$.
\end{definition}

On pose
\begin{equation}
  F_T^{(n)} = \sum_{c\in \coul(T,n)} q^{\sigma(c)},
\end{equation}
où la somme porte sur les $n$-coloriages décroissants de $T$ et
$\sigma(c)$ est la somme des valeurs de $c$.

Ces polynômes sont déterminés par la relation
\begin{equation}
  \label{calcul_Fq}
  F^{(n)}_{\bop(T_1,\dots,T_k)}= \sum_{j=0}^{n} q^j \prod_{i=1}^k F^{(j)}_{T_i}.
\end{equation}
En effet, si la couleur de la racine est $j$, les sous-arbres $T_i$ sont
coloriés par des entiers inférieurs ou égaux à $j$.

Par exemple, le polynôme $F^{(1)}_T$ pour l'arbre $T$ de la figure
\ref{fig:arbre5} est
\begin{equation}
  q^5 + 3 q^4 + 3 q^3 + 2 q^2 + q + 1.
\end{equation}

Soit $F^{(n)}$ la série en arbres définie par
\begin{equation}
  F^{(n)} = \sum_{T} F_T^{(n)} \frac{T}{\aut(T)}.
\end{equation}

On a
\begin{equation}
  \label{prop_Fn}
  F^{(n)} = E + \COR \diam (q\Sigma_q F^{(n-1)}, E).
\end{equation}

En effet, on distingue les sommets de couleur $0$ et le reste. Soit
tous les sommets ont couleur $0$, et on a la série $E$, soit on a la
greffe d'arbres de couleur $0$ sur un arbre colorié par
$\{1,\dots,n\}$. Ceci donne exactement le second terme du membre de droite.

Par ailleurs, la proposition \ref{prop_gen} et \eqref{prop_Fn}
entraînent que
\begin{equation}
  \label{auxi_Fn_1}
  \COR \diam (F^{(n)}, -\pun) = \pun + q\Sigma_q F^{(n-1)} .
\end{equation}

Considérons la série $q \Sigma_q F^{(n)}$. Le coefficient d'un arbre
$T$ dans cette série est la somme
\begin{equation*}
    q\Sigma_q F_T^{(n)} = \sum_{c} q^{\sigma(c)}
\end{equation*}
portant sur les coloriages décroissants par $\{1,\dots,n+1\}$. On distingue
deux types de tels coloriages. Si la couleur de la racine est $n+1$,
les sous-arbres sont coloriés par $\{1,\dots,n+1\}$. Sinon, l'arbre
$T$ est colorié par $\{1,\dots,n\}$. On obtient la formule
\begin{equation}
  \label{auxi_Fn_2}
  q \Sigma_q F^{(n)} = q \Sigma_q F^{(n-1)} + q^{n+1} \COR \diam (\pun, q\Sigma_q F^{(n)} ).
\end{equation}

On déduit des égalités \eqref{auxi_Fn_1} et \eqref{auxi_Fn_2} que $F$
est solution de \eqref{main_1}.

Par unicité de la solution, on obtient
\begin{theorem}
  \label{valeur_n_positif}
  La valeur de $\sqx$ en $x=[n]_q$ est la série $F^{(n)}$, dont les
  coefficients comptent les coloriages décroissants par $\{0,\dots,n\}$.
\end{theorem}

\subsection{Valeur en $x=1/(1-q)$}

On peut passer à la limite des coloriages lorsque $n=\infty$. Ceci
correspond à spécialiser $x$ en $1/(1-q)$. L'équation \eqref{main}
devient alors
\begin{equation}
  \label{main_infini}
  q \Sigma_q Z - \COR \diam (Z, -\pun) = - \pun,
\end{equation}
pour la série $Z = \sqx\mid_{x=1/(1-q)}$.

La solution $Z$ a pour coefficients des
fractions en $q$ ayant des pôles en $q=1$. Le coefficient d'un arbre
$T$ est la fraction définie par la série formelle
\begin{equation*}
 \sum_{c\in \coul(T)} q^{\sigma(c)},
\end{equation*}
la somme portant sur tous les coloriages décroissants de $T$. En
particulier, pour le coefficient de la corolle $\corol_k$, on trouve
l'expression
\begin{equation}
  \sum_{j=1}^{\infty} q^{j-1} [j]_q^k.
\end{equation}

On rappelle l'opérateur $\zeta_q(-k)$ (agissant
sur les séries formelles en $q$ sans terme constant) introduit dans
\cite{chap_zeta}, défini par
\begin{equation*}
  \zeta(-k)(f) = \sum_{j=1}^{\infty} f(q^j) [j]_q^k.
\end{equation*}

Le coefficient de la corolle $\corol_k$ peut s'exprimer comme
\begin{equation}
  \frac{1}{q} \zeta_q(-k)(q) .
\end{equation}

\subsection{Évaluation en $x=1$ puis $q=-1$}

On obtient une série intéressante en spécialisant en $q=-1$ la série
$F^{(1)}$. Notons $\overline{F}$ la série obtenue.

Lorsque $n=1$ et $q=-1$, la relation \eqref{calcul_Fq} devient
\begin{equation}
  \overline{F}_{\bop(T_1,\dots,T_k)}= 1 - \prod_{i=1}^k \overline{F}_{T_i}.
\end{equation}

Cette relation entraîne par récurrence que $\overline{F}_T \in
\{0,1\}$ pour tout $T$.

On dit qu'un arbre $T$ est de type $0$ (resp. $1$) si
$\overline{F}_T=0$ (resp. $1$).

Un arbre est de type $0$ si et seulement si il est de la forme
$\bop(T_1,\dots,T_k) $ avec tous les $T_i$ de type $1$. Un arbre est de
type $1$ si et seulement si il est de la forme $ \bop(T_1,\dots,T_k) $
avec au moins un $T_i$ de type $0$.

Il y a une relation entre cette partition de l'ensemble des arbres
enracinés en deux types et les résultats de \cite{bauer} portant sur
un tri-coloriage canonique des arbres non-enracinés. Les arbres
enracinés de type $0$ sont exactement les arbres enracinés dont la
racine est verte au sens de cet article. On peut en déduire la
description suivante des arbres de type $0$ et $1$.

On appelle \textbf{couverture minimale par sommets} d'un arbre $T$ un
ensemble $C$ de sommets de $T$, tel que chaque arête ait au moins une
des ses extrémités dans $C$, et de cardinal minimal pour cette
propriété.

Alors un arbre est de type $1$ si et seulement si et seulement si il
existe une couverture minimale par sommets contenant la racine, et de
type $0$ si et seulement si aucune couverture minimale par sommets ne
contient la racine. Par exemple, l'arbre de la figure \ref{fig:arbre5}
est de type $1$.

\subsection{Valeur en $x=[n]_q$ pour $n$ négatif}

Lorsque $n$ est un entier négatif, l'équation \eqref{main_1} reste
valable. En remplaçant dans cette équation $q$ par $1/q$ et $n$ par
$-n$, on obtient, pour tout entier $n$ positif, l'équation
\begin{equation}
  \label{main_neg}
  1/q \Sigma_{1/q} Z - \COR \diam (Z, -\pun) = q^{n-1} \COR \diam (\pun, 1/q\Sigma_{1/q} Z ) - \pun,
\end{equation}
pour la série $Z$ obtenue en remplaçant $q$ par $1/q$ dans
$\sqx\mid_{x=[n]_q}$. Pour $n=1$, la seule
solution est identiquement nulle, comme on l'a déjà dit plus haut. On
suppose maintenant que $n\geq 2$.

En posant $Z = (-q) \Sigma_{-q} Y$, on trouve
\begin{equation}
  \label{autre_neg}
  - Y + \COR \diam ( q\Sigma_q Y, \pun) = q^{n-1} \COR \diam (\pun, Y ) - \pun.
\end{equation}

Soit $\leq$ la relation d'ordre partiel sur les sommets de $T$ ayant
la racine pour minimum.

\begin{definition}
  Un $n$-coloriage décroissant strict de $T$ est une application
  strictement décroissante $c$ de $T$ dans l'ensemble $\{0,\dots,n\}$.
\end{definition}

Pour chaque arbre $T$ et un entier $n$, on définit un polynôme
$G_T^{(n)}$ en $q$ par la formule
\begin{equation}
  G_T^{(n)} = \sum_{c\in \coul(T,n,<)} q^{\sigma(c)},
\end{equation}
la somme portant sur les $n$-coloriages décroissants stricts de $T$.

Ces polynômes sont déterminés par la relation
\begin{equation}
  \label{calcul_Gq}
  G^{(n)}_{\bop(T_1,\dots,T_k)}= \sum_{j=0}^{n} q^j \prod_{i=1}^k G^{(j-1)}_{T_i}.
\end{equation}
En effet, si la couleur de la racine est $j$, les sous-arbres $T_i$
sont coloriés par des entiers strictement inférieurs à $j$.

Par exemple, le polynôme $G^{(3)}_T$ pour l'arbre $T$ de la figure
\ref{fig:arbre5} est
\begin{equation}
  q^9 + 3 q^8 + 4 q^7 + 4 q^6 + 2 q^5 + 2 q^4 + q^3.
\end{equation}

Soit $G^{(n)}$ la série définie par
\begin{equation}
  G^{(n)} = \sum_{T} G_T^{(n)} \frac{T}{\aut(T)}.
\end{equation}

On a
\begin{equation}
  \label{prop_G_1}
  G^{(n+1)} - \pun = \COR \diam ( q\Sigma_q G^{(n)}, \pun).
\end{equation}
En effet, le terme de gauche correspond aux coloriages décroissants
stricts par $\{0,\dots,n+1\}$ dont la racine ne porte pas la couleur
$0$. Le terme de droite correspond à la greffe de sommets isolés
portant la couleur $0$ sur un coloriage décroissant strict par
$\{1,\dots,n+1\}$.

Par ailleurs, on a
\begin{equation}
  \label{prop_G_2}
  G^{(n+1)} - G^{(n)} = q^{n+1} \COR \diam ( \pun, G^{(n)}).
\end{equation}
En effet, le terme de gauche correspond aux coloriages décroissants
stricts par $\{0,\dots,n+1\}$ dont la racine porte la couleur $n+1$. Le terme de droite correspond à la greffe sur une racine portant la couleur $n+1$ de coloriages décroissants stricts par $\{0,\dots,n\}$.

On déduit des égalités \eqref{prop_G_1} et \eqref{prop_G_2} que $G^{(n-2)}$
est solution de \eqref{autre_neg}.

Par unicité de la solution, on obtient
\begin{theorem}
  \label{valeur_n_negatif}
  La valeur de $ \sqx$ en $x=[-n]_q$, évaluée en $q=1/q$, est la série
  $-q \Sigma_{-q} G^{(n-2)}$.
\end{theorem}

On en déduit le corollaire suivant.

\begin{proposition}
  \label{facteurs_connus}
  Soit $H$ la hauteur de $T$. Le numérateur de $\sqx_T$ est divisible
  par le produit
  \begin{equation}
    \prod_{i=1}^{H} \left([i]_q+q^i x\right).
  \end{equation}
\end{proposition}
\begin{proof}
  Soit $i$ un entier entre $1$ et $H$. Comme $i-1$ est strictement
  inférieur à la hauteur de l'arbre, il n'existe aucun coloriage
  décroissant strict par $\{0,\dots,i-2\}$. Par conséquent, le
  coefficient $\sqx_T$ s'annule en $x=[-i]_q$, ce qui entraîne le résultat.
\end{proof}

\subsection{Limite en $q=1$}

\begin{proposition}
  Pour tout $T$, la fraction $\sqx_T$ n'a pas de pôle en $q=1$.
\end{proposition}
\begin{proof}
  Les évaluations en les $q$-entiers $[n]_q$ pour $n\in \ZZ_{\geq 0}$ sont des
  polynômes en $q$. Comme le degré de $\sqx_T$ par rapport à la
  variable $x$ est borné par la taille de $T$, on peut retrouver
  $\sqx_T$ par interpolation. Il résulte de la formule d'interpolation
  que $\sqx_T$ est bien défini en $q=1$.
\end{proof}

On peut déduire de \eqref{recu_main} que la limite $Z$ de $\sqx$ en $q=1$
vérifie
\begin{equation}
  \label{q_egal_1}
  0 = \left(\COR \diam (Z, -\pun) - Z\right) + \COR \diam (\pun,Z ) - \pun.
\end{equation}

\begin{proposition}
  La condition initiale $Z_1=1+x$ et l'équation \eqref{q_egal_1}
  caractérisent la série $Z$.
\end{proposition}
\begin{proof}
  L'équation \eqref{q_egal_1} peut s'écrire sous la forme
  \begin{equation*}
    [Z,\pun] = \left(\COR \diam (Z, -\pun) - Z + Z \pl \pun \right) + \left(\COR \diam (\pun,Z ) - \pun - \pun \pl Z\right),
  \end{equation*}
  où le membre de gauche fait intervenir le crochet de Lie associé au
  produit pré-Lie $\pl$. Soit $n\geq 2$. La composante homogène de degré
  $n+1$ de cette équation permet d'exprimer la composante de degré
  $n+1$ de $[Z,\pun]$ en fonction des composantes $Z_k$ pour $k \leq
  n-1$. Il suffit alors d'utiliser que l'application $x\mapsto
  [x,\pun]$ est injective lorsque le degré de $x$ est au moins $2$, ce
  qui résulte du fait que les algèbres pré-Lie libres sont libres en
  tant qu'algèbres de Lie \cite{foissy,chap_libre}.
\end{proof}

\subsection{Valeur limite en $x=-1/q$}

On rappelle la série $\Omega_q$ introduite et étudiée dans l'article
\cite{chap_omega}. Elle est définie par l'équation
\begin{equation}
  \COR \diam (q\Sigma_q \Omega_q,\pun) -\Omega_q = \pun \pl \Omega_q +(q-1)
\pun.
\end{equation}

On introduit une variante $\baro_q$ définie par
\begin{equation}
  \label{def_baroq}
  \baro_q = \Sigma_{(-1/q)} \Omega_{1/q}.
\end{equation}

On montre aisément que cette variante vérifie l'équation
\begin{equation}
  \label{defi_baro}
  q\Sigma_q \baro_q - \COR \diam (\baro_q, -\pun) = (q-1) \pun + q \,\pun \pl \Sigma_q \baro_q,
\end{equation}
qui la caractérise.

Par la proposition \ref{facteurs_connus}, tous les coefficients de
$\sqx$ sont divisibles par $1+q x$.

On considère ici la limite de $\frac{1}{1+q x}\sqx$ lorsque $x=-1/q$.

\begin{theorem}
  \label{valeur_speciale}
  La limite de $\frac{1}{1+q x}\sqx$ en $x=-1/q$ est la série $\baro_q$.
\end{theorem}
\begin{proof}
  On trouve exactement comme limite pour l'équation \eqref{main}
  l'équation \eqref{defi_baro} vérifiée par $\baro_q$.
\end{proof}

Par exemple, on trouve ainsi que le coefficient de l'arbre de la
figure \ref{fig:arbre5} dans la série $\baro_q$ est la fraction
\begin{equation}
  \frac{1 + q - q^3}{\Phi_2\Phi_3\Phi_4\Phi_5}.
\end{equation}

D'après \cite{chap_omega}, on sait que les dénominateurs des
coefficients de la série $\Omega_q$ sont des produits de polynômes
cyclotomiques sans multiplicité. C'est donc vrai aussi pour la
variante $\baro_q$. Par contre, c'est faux pour la série $\sqx$, par
exemple pour l'arbre de la figure \ref{fig:arbre5}.

Il se produit donc nécessairement des simplifications lors de la
limite en $x=-1/q$. Certaines de ces simplifications font intervenir
les facteurs du numérateur de $\sqx_T$ décrits dans la proposition
\ref{facteurs_connus}. Mais ceci ne suffit pas à expliquer toutes les
simplifications nécessaires, comme on le voit en considérant l'exemple
de $\bop(\corol_2,\pun,\pun)$.


\subsection{Valeur en $x=\infty$}

Par valeur en $x=\infty$, on entend la série obtenue en ne gardant que
le terme homogène de degré $n$ en $x$ dans chaque composante $\sqx_n$
de $\sqx$.

Par un passage à la limite convenable dans \eqref{main}, on obtient
pour cette valeur l'équation
\begin{equation}
  q \Sigma_q Z - Z = q (q-1) \COR \diam (\pun, q\Sigma_q Z),
\end{equation}
pour la série $Z = \sqx\mid_{x=\infty}$. Le terme de degré $1$ de $Z$ est $q \,\pun$.

On peut écrire ceci sous la forme
\begin{equation}
  \label{recu_infty}
  \frac{  q \Sigma_q Z - Z}{q-1} = q \COR \diam (\pun, q\Sigma_q Z).
\end{equation}

\begin{definition}
  La $q$-factorielle d'un arbre $T$ est définie comme suit :
  \begin{equation}
    [T]!_q = q^{-\sum_{v \in T}\# T_v}\prod_{v \in T} [\# T_v]_q
  \end{equation}
  où $T_v$ est le sous arbre de $T$ associé au sommet $v$.
\end{definition}

En $q=1$, on retrouve la factorielle d'un arbre usuelle, définie par
\begin{equation}
  T! = \prod_{v \in T} \# T_v.
\end{equation}

Par exemple, la $q$-factorielle de l'arbre $T$ de la figure \ref{fig:arbre5} est
\begin{equation}
  q^{-1-1-1-3-5} [1]_q [1]_q [1]_q [3]_q [5]_q = q^{-11} [3]_q [5]_q.
\end{equation}

\begin{proposition}
  La valeur de $\sqx$ en $x=\infty$ est la série
  \begin{equation}
    \sum_{T} \frac{1}{[T]!_q} \frac{T}{\aut(T)}.
  \end{equation}
\end{proposition}
\begin{proof}
  L'énoncé est vrai en degré $1$ par inspection. La composante
  homogène de \eqref{recu_infty} de degré $n$ s'écrit
  \begin{equation}
    [n]_q Z_n = q \COR \diam (\pun, q\Sigma_q Z),
  \end{equation}
  en ne gardant à droite que la partie de degré $n$.

  Par le lemme \ref{lemme_corolle}, ceci est équivalent à la relation
  suivante entre les coefficients
  \begin{equation}
    [\#T]_q Z_T = q \prod_{i=1}^{k} q^{\# T_i} Z_{T_i},
  \end{equation}
  lorsque $T= \bop(T_1,\dots,T_k)$. Cette relation est exactement la
  définition des inverses des $q$-factorielles des arbres.
\end{proof}

\begin{remark}
  En particulier, le degré de $\sqx_T$ est toujours exactement $\# T$.
\end{remark}

\section{Propriétés ombrales}

\label{autres}

Cette section est consacrée à la description d'une relation, distincte
de celle du théorème \ref{valeur_speciale}, entre les séries
$\sqx$ et $\baro_q$, par le biais de certains $q$-analogues des
nombres de Bernoulli.

\subsection{Opérateur de Hahn et opérateur $\bop$}

On introduit l'opérateur $\Delta$ :
\begin{equation}
  \Delta(f) = \frac{f(1+q x)-f(x)}{1+ q x -x},
\end{equation}
agissant sur les fonctions de $x$ à coefficients dans $\QQ(q)$. Il est
$\QQ(q)$-linéaire.

Cet opérateur, qui est une forme de $q$-dérivation, a notamment été
considéré par Hahn \cite{hahn}.

Soit $f$ un polynôme en $x$ à coefficients dans $\QQ(q)$ et notons
$f_n=f([n]_q)$. Alors
\begin{equation}
  \label{delta_val}
  \Delta(f)([n]_q) = \frac{f_{n+1}-f_{n}}{q^n}.
\end{equation}

\begin{proposition}
  \label{action_delta}
  Si $T=\bop(T_1,\dots,T_k)$, on a
  \begin{equation}
    \Delta (\sqx_T) = q \prod_{i=1}^{k} \sqx_{T_i}(1+q x).
  \end{equation}
\end{proposition}
\begin{proof}
  Il s'agit de montrer une égalité entre deux polynômes en $x$. Il
  suffit de montrer l'égalité de leurs valeurs en $[n]_q$ pour tout
  entier positif $n$.

  Par la description \eqref{delta_val} de l'action de $\Delta$ sur les
  valeurs en $[n]_q$ et par le théorème \ref{valeur_n_positif}, on trouve
  que $q^n \Delta (\sqx_T)([n]_q)$ compte les coloriages décroissants
  de $T$ par les entiers entre $0$ et $n+1$, sous la condition de
  contenir $n+1$.

  Pour compter ces coloriages, on peut aussi multiplier le produit des
  valeurs de $\sqx_{T_i}$ en $[n+1]_q$ par le facteur $q^{n+1}$
  correspondant au coloriage de la racine de $T$ par $n+1$. On a donc
  \begin{equation*}
    q^n \Delta (\sqx_T)([n]_q) = q^{n+1} \prod_{i=1}^{k} \sqx_{T_i}([n+1]_q).
  \end{equation*}

  On en déduit le résultat.
\end{proof}

Considérons l'action de $\Delta$ sur l'espace $\QQ(q)[x]$ des
polynômes en $x$ à coefficients dans $\QQ(q)$. Le noyau de $\Delta$
est le sous-espace $\QQ(q)$ des polynômes constants. Un supplémentaire
est fourni par les polynômes qui sont divisibles par $1+q x$. La
restriction de $\Delta$ à ce supplémentaire est un isomorphisme avec
l'espace $\QQ(q)[x]$.

\subsection{Ombre de Bernoulli}

On rappelle que les nombres de Bernoulli-Carlitz \cite{carlitz} sont
des fractions en $q$ définies par $\beta_0=1$ et
\begin{equation}
  \label{recu_beta}
  q(q \beta +1)^n-\beta_n=
  \begin{cases}
    1 \text{ si }n=1,\\
    0 \text{ si }n>1,
   \end{cases}
\end{equation}
où par convention on remplace $\beta^k$ par $\beta_k$ après avoir
développé la puissance du binôme.

Soit $P$ un polynôme en $x$ à coefficient dans $\QQ(q)$. On appelle
$q$-ombre de $P$ la valeur en $P$ de la forme $\QQ(q)$-linéaire qui
envoie $x^n$ sur le nombre de Bernoulli-Carlitz $\beta_n$. On note
$\Psi(P)$ la $q$-ombre de $P$, qui est une fraction en $q$.

Soient $T_1,\dots,T_k$ des arbres enracinés et soit $T=\bop(T_1,\dots,T_k)$.
\begin{theorem}
  \label{ombral_iti}
  On a
  \begin{equation}
    \label{ombral_id}
    \baro_{q,T} = \Psi\left(\prod_{i=1}^{k} \sqx_{T_i}\right).
  \end{equation}
\end{theorem}

\begin{proof}
  L'opérateur $\Psi$ est un opérateur $\QQ(q)$-linéaire, agissant sur
  l'espace vectoriel des polynômes en $x$. On va montrer que le côté
  gauche de cette égalité est aussi la valeur d'un opérateur
  $\QQ(q)$-linéaire agissant sur le même argument, et comparer ensuite
  leurs valeurs sur les puissances de $1+q x$.

  Par le théorème \ref{valeur_speciale}, le côté gauche est la
  valeur limite de $\sqx_T/(1+q x)$ lorsque $x=-1/q$. L'application qui
  associe à un polynôme $P$ divisible par $1+ q x$ la limite de
  $P/(1+q x)$ lorsque $x=-1/q$ est évidemment linéaire.

  Par la proposition \ref{action_delta}, on a
  \begin{equation*}
    \sqx_T = \Delta^{-1} \left( q \prod_{i=1}^{k} \sqx_{T_i}(1+q x) \right),
  \end{equation*}
  où $\Delta^{-1}$ est l'opérateur inverse de la restriction de
  $\Delta$ aux polynômes divisibles par $1 +q x$.

  Les deux côtés de \eqref{ombral_id} sont donc des opérateurs
  $\QQ(q)$-linéaires sur l'espace vectoriel des polynômes en $x$. Il
  suffit donc de comparer leurs valeurs sur la base de cet espace
  donnée par les puissances $(1+q x)^k$ pour $k\geq 0$.

  Comme $\sqx_{\pun} = 1+q x$, ceci revient à vérifier l'identité
  voulue lorsque tous les $T_i$ sont égaux à $\pun$. Par la définition
  de $\Psi$, et par \eqref{recu_beta}, on obtient à droite $1$ si
  $k=0$, $1/(q+1)$ si $k=1$ et $\beta_k/q$ si $k\geq 2$. 

  On vérifie aisément que les coefficients de $\pun$ et de $\corol_1$
  dans $\baro_q$ sont $1$ et $1/(1+q)$. Par \cite[(5.11)]{carlitz}, on a
  \begin{equation*}
    \beta_k(1/q) = (-1)^k q^{k-1} \beta_k,
  \end{equation*}
  pour $k \geq 2$. Comme le coefficient de $\corol_k$ dans $\Omega_q$
  est $\beta_k$, on déduit de \eqref{def_baroq} que le coefficient de
  $\corol_k$ dans $\baro_q$ est aussi donné par $\beta_k/q$ si $k\geq
  2$.
\end{proof}

\begin{remark}
  \label{linear_baro}
  Dans l'article \cite{chap_omega}, on a montré que le coefficient de
  $\linear_n$ dans la série $\Omega_q$ est $(-1)^{n-1}/[n]_q$. Le
  coefficient de $\linear_n$ dans la série $\baro_q$ est donc
  $1/[n]_q$. 
\end{remark}

\begin{lemma}
  \label{petit_calcul_ombral}
  Pour tout $n \geq 1$, on a
  \begin{equation}
    \Psi(-x \sqx_{\linear_{n}}) = \frac{1}{[n+2]_q}.
  \end{equation}
\end{lemma}
\begin{proof}
  Par la remarque \ref{linear_baro} et le théorème
  \ref{ombral_iti}, on déduit que
  \begin{equation*}
    \Psi(\sqx_{\linear_{n}}) = \frac{1}{[n+1]_q}.
  \end{equation*}
  Par ailleurs, en se servant de la forme explicite de
  $\sqx_{\linear_{n+1}}$ donnée par \eqref{formule_linear}, on peut
  montrer que
  \begin{equation*}
    \frac{q^{n+1}}{[n+1]_q} \Psi(x \sqx_{\linear_{n}}) =  \Psi(\sqx_{\linear_{n+1}})-\Psi(\sqx_{\linear_{n}}).
  \end{equation*}
  Le résultat voulu en découle.
\end{proof}

\begin{theorem}
  \label{ombral_nui}
  On a
  \begin{equation}
    \baro_{q,\bop T} = \Psi\left( - x \prod_{i=1}^{k} \sqx_{T_i}\right).
  \end{equation}
\end{theorem}
\begin{proof}
  Pour les mêmes raisons que dans la preuve du théorème
  \ref{ombral_iti}, les deux membres de cette égalité sont les valeurs
  en le polynôme $\prod_{i=1}^{k} \sqx_{T_i}$ de deux opérateurs
  $\QQ(q)$-linéaires. Il suffit donc de les comparer sur la base
  formée par les polynômes $\sqx_{\linear_n}$ pour $n\geq 1$ et par le
  polynôme constant $1$.

  Pour le polynôme constant $1$, on constate que l'opposé du nombre de
  Bernoulli-Carlitz $\beta_1$ est $1/(1+q)$, qui est exactement le
  coefficient de l'arbre $\linear_2$ dans $\baro_q$.

  Lorsqu'on prend pour seul $T_i$ l'arbre $\linear_n$, le membre de
  gauche est $1/[n+2]_q$ par la remarque \ref{linear_baro}. Le membre
  de droite a la même valeur par le lemme \ref{petit_calcul_ombral}.
\end{proof}

\section{Conjectures}

\label{conjectures}

On présente dans cette section deux conjectures sur les séries $\sqx$
et $\Omega_q$.

\subsection{Polygones de Newton des numérateurs}

\begin{figure}[h!]
  \centering
  \includegraphics[height=2cm]{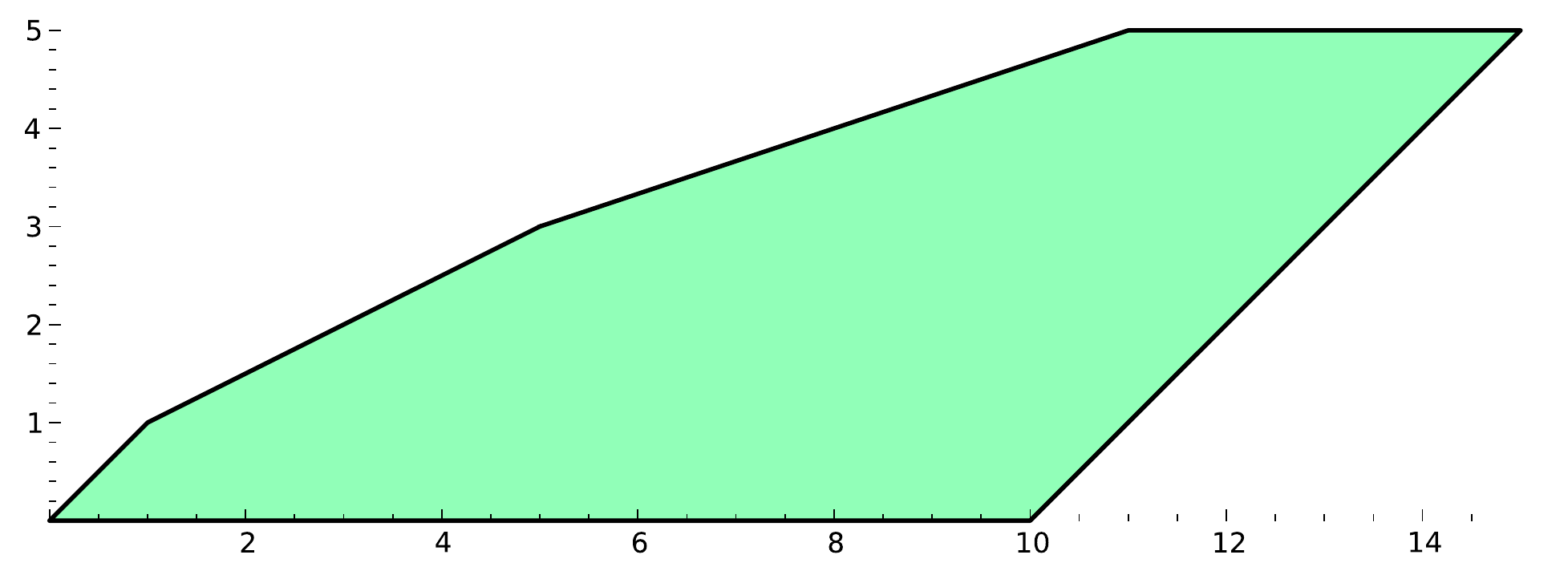}
 \caption{Polygone de Newton du numérateur d'un coefficient $\sqx_T$.}
  \label{fig:newton}
\end{figure}

Soit $\square_T$ le polygone de Newton du numérateur de $\sqx_T$, voir
la figure \ref{fig:newton} pour celui associé à l'arbre de la figure
\ref{fig:arbre5}.

Il semble que la forme de $\square_T$ soit comme suit.

\begin{conjecture}
  \label{conj_newton}
  Le bord supérieur de $\square_T$ est une droite horizontale
  correspondant au coefficient de degré $\# T$ par rapport à la
  variable $x$. Le bord inférieur est une droite horizontale
  correspondant au coefficient constant par rapport à la variable
  $x$. Le bord droit est une droite de pente $1$. Le bord gauche est
  une suite de segments de pente $1/i$ pour tout $i$ entre $1$ et la
  hauteur de $T$. Le segment de pente $1/i$ a pour hauteur le nombre
  de sommets de hauteur $i$ dans $T$.
\end{conjecture}





\subsection{Arbres en forme de partitions}

On propose ici une conjecture sur les coefficients de $\Omega_q$.

A chaque partition $\lambda=(\lambda_1,\dots,\lambda_m)$ d'un entier
$n$, on associe l'arbre $T_\lambda$ à $n+1$ sommets défini par
\begin{equation*}
  T_\lambda = \bop (\linear_{\lambda_1},\dots,\linear_{\lambda_m}).
\end{equation*}

\begin{conjecture}
  Pour tout entier $k\geq 3$ impair, le coefficient $\Omega_{q,\bop(T_\lambda^k)}$
  est divisible par $\Phi_{1+\max\lambda}$ .
\end{conjecture}

Dans le cas de la partition vide, l'arbre $T_\lambda$ est $\pun$ et
les arbres $\bop(T_\lambda^k)$ sont les corolles. Les coefficients des
corolles dans $\Omega_q$ sont les nombres de Bernoulli-Carlitz, qui
donnent les nombres de Bernoulli usuels en $q=1$. La conjecture est
donc bien connue dans ce cas.




\section{Résultat auxiliaire}

Soit $E$ la série en arbres définie par \eqref{defi_E}.

\label{auxiliaire}



\begin{proposition}
  \label{prop_gen}
  Pour toute série en arbres $A$, on a
  \begin{equation}
    \COR \diam (\COR \diam (A, E)),-\pun) = A.
  \end{equation}
\end{proposition}

\begin{proof}
  En utilisant la formule \eqref{associatif}, il s'agit de calculer
  l'expression équivalente
  \begin{equation*}
    \COR \diam (A, \COR \diam (E,-\pun)).
  \end{equation*}
  On va considérer, pour $k$ entier positif, l'expression
  \begin{equation*}
    \COR \diam (E,k\,\pun).
  \end{equation*}
  Les coefficients de cette série comptent les arbres dont chaque
  feuille est soit vide, soit décorée par une couleur parmi $k$. On
  peut aussi voir ceci comme des arbres dont toutes les feuilles sont
  décorées par une couleur parmi $k+1$. Par conséquent, tous les
  coefficients sont divisibles par $k+1$. Donc
  \begin{equation*}
    \COR \diam (E, - \,\pun) = 0
  \end{equation*}

  Il en résulte que
  \begin{equation*}
    \COR \diam (A, \COR \diam (E,- \,\pun)) = A.
  \end{equation*}
\end{proof}

\appendix
\section{Termes initiaux}

\label{debut}

La série $\sqx$ commence par
\begin{multline*}
  (1+q x) \pun + \frac{(1+q x)(1+q +q^2 x)}{\Phi_2}\arb{10} \\ +\frac{(1+q x)(1+q +q^2 x)(1+q +q^2 +q^3 x)}{\Phi_2\Phi_3}\arb{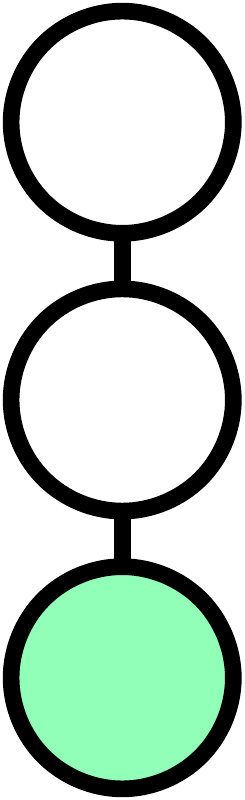} \\ + \frac{(1+q x)(1+q +q^2 x)(1+q +q^2 +q^2 x + q^3 x)}{\Phi_2\Phi_3} \frac{\arb{200}}{2} +\,\cdots
\end{multline*}

La série $\baro_q$ commence par
\begin{multline*}
\baro_q=\pun+\frac{1}{\Phi_2}\arb{10}+\frac{1}{\Phi_3} 
\arb{110}+\frac{1}{\Phi_2 \Phi_3}\frac{\arb{200}}{2} \\
+\frac{1}{\Phi_2 \Phi_4}
\arb{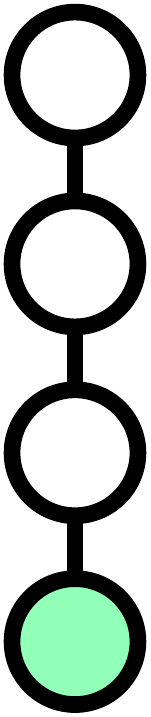}+\frac{1}{\Phi_3 \Phi_4}\frac{\arb{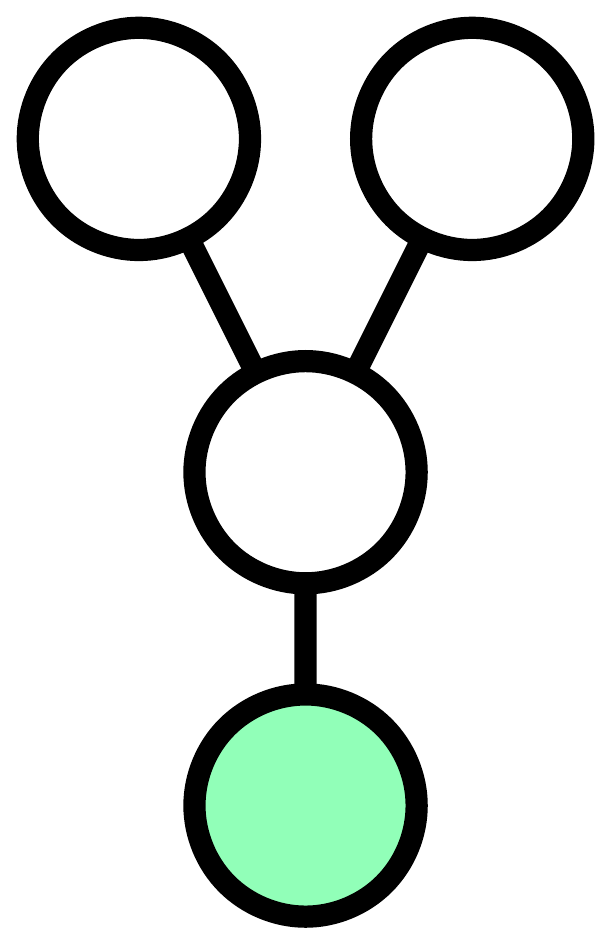}}{2}+\frac{1}{\Phi_2 \Phi_3
\Phi_4}\arb{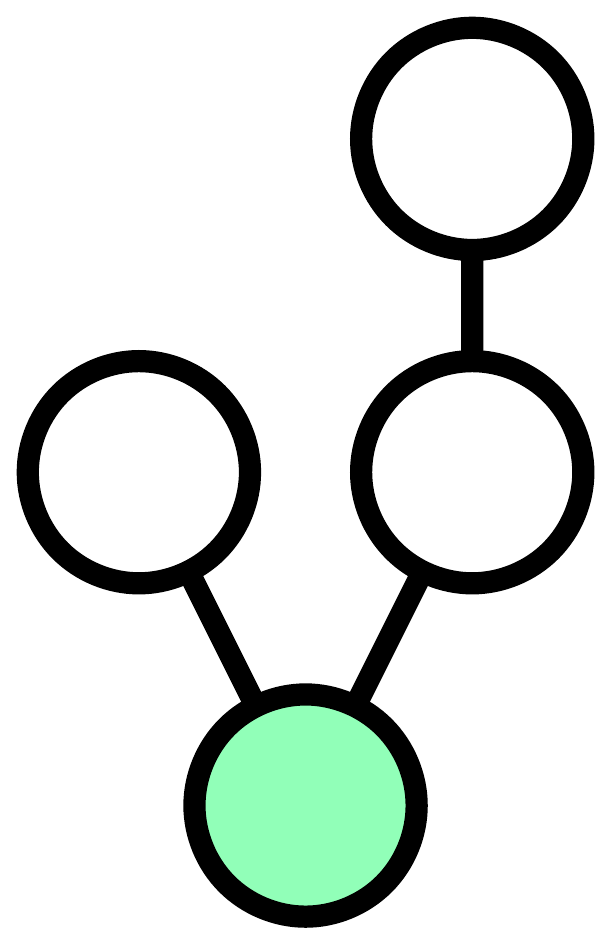} +\frac{1-q}{\, \Phi_2\Phi_3\Phi_4}\frac{\arb{3000}}{6}
+\,\cdots
\end{multline*}

\bibliographystyle{plain}
\bibliography{serieqx}

\end{document}